\pdfoutput=1
\documentclass[leqno]{article}
\usepackage[normal]{phaine_style}
\usepackage[margin=1.5in]{geometry}
\usepackage{setspace}

\tikzcdset{
  cells={font=\everymath\expandafter{\the\everymath\displaystyle}},
}

\newcommand{\Frame}{\categ{Frm}}
\newcommand{\Latt}{\categ{Lat}}

\newcommand{\Modbf}{\categ{Mod}}

\newcommand{\GalL}{\mathrm{Gal}_{\mathrm{L}}}
\newcommand{\cond}{\mathrm{cond}}
\newcommand{\Picond}{\Pi_{\infty}^{\cond}}

\newcommand{\Pretop}{\categ{Pretop}_{\infty}}
\newcommand{\Pretopb}{\Pretop^{\b}}

\newcommand{\cohtrun}{_{<\infty}^{\coh}}
\renewcommand{\b}{\mathrm{b}}

\newcommand{\Syn}{\mathrm{Syn}}

\newcommand{\AN}{\categ{\textsc{An}}}
\newcommand{\CATinfty}{\categ{\textsc{Cat}}_{\infty}}

\newcommand{\Extr}{\categ{Extr}}

\newcommand{\Acc}{\categ{Acc}}
\newcommand{\fil}{\ensuremath{\textup{fil}}}
\newcommand{\Accfil}{\Acc{}^{\fil}}

\newcommand{\bc}{\mathrm{bc}}
\newcommand{\Cond}{\mathrm{Cond}}
\newcommand{\RTop}{\categ{RTop}_{\infty}}
\newcommand{\RTopbc}{\RTop^{\bc}}
\newcommand{\RTopcoh}{\RTop^{\coh}}
\newcommand{\spec}{\mathrm{spec}}
\newcommand{\RTopspec}{\RTop^{\spec}}

\newcommand{\idem}{\mathrm{idem}}
\newcommand{\Catidem}{\Catinfty^{\idem}}

\renewcommand{\ni}{\smallni}
\newcommand{\Kcolim}{\ensuremath{\Kcal\textup{-colim}}}

\newcommand{\Catfin}{\Cat_{\infty,\uppi}}
\newcommand{\Layfin}{\Lay_{\uppi}}

\newcommand{\Clopen}{\mathrm{Clopen}}
\renewcommand{\Open}{\mathrm{Open}}

\newcommand{\dmu}{\kern0.05em\mathrm{d} \mu}

\renewcommand{\coh}{\mathrm{coh}}
\newcommand{\Ptbf}{\mathbf{Pt}}
\newcommand{\Ptcohbf}{\Ptbf^{\coh}}

\renewcommand{\pre}{\mathrm{pre}}
\newcommand{\Funpre}{\Fun^{\pre}}
\newcommand{\Funupperstarcoh}{\Fun^{\ast,\coh}}

\newcommand{\Sheff}{\Sh_{\eff}}

\newcommand{\An}{\categ{An}}

\usepackage{caption}
\renewcommand{\Spc}{\Ani}

\DeclareSourcemap{
  \maps[datatype=bibtex]{
    \map{
      \step[fieldsource=note, final]
      \step[fieldset=addendum, origfieldval, final]
      \step[fieldset=note, null]
    }
  }
}

\title{\Large Classifying anima of condensed \texorpdfstring{$\infty$}{∞}-categories of points}

\author{\normalsize Peter J. Haine}
 
\date{\normalsize \today}

\begin{document}

\maketitle


\begin{abstract} 
	We compare the classifying anima of two natural condensed \categories associated to a coherent \topos.
	One from our work with Barwick and Glasman on exit-path categories in algebraic geometry, and the other from Lurie's work on ultracategories.
	The key consequence of our comparison is a connection between algebraic geometry and model theory: up to a mild completion, the proétale fundamental group of a scheme and the Lascar group of a complete first-order theory are both special cases of the same construction.
\end{abstract}

\setcounter{tocdepth}{1}

\tableofcontents


\setcounter{section}{-1}

\section{Introduction}

The purpose of this note is to compare the classifying anima of two natural condensed \categories associated to a coherent \topos.
The first condensed \category makes sense for any \topos $ \Xcal $ and was introduced (in the setting of $ 1 $-topoi) in Lurie's work on ultracategories \cite{Ultracategories}.
We call it the \defn{condensed \category of points} $ \Ptbf(\Xcal) $ of $ \Xcal $; it is given by sending an extremally disconnected profinite set $ K $ to the \category
\begin{equation*}
	\Ptbf(\Xcal)(K) \colonequals \Funupperstar(\Xcal,\Sh(K))
\end{equation*}
of left exact left adjoints $ \fupperstar \colon \Xcal \to \Sh(K) $.
In particular, the global sections of $ \Ptbf(\Xcal) $ recover the \category of points of $ \Xcal $. 

The second condensed \category we consider only makes sense when $ \Xcal $ is a \textit{coherent} \topos (the topos-theoretic version of quasicompactness and quasiseparatedness in scheme theory), and was implicitly introduced in our work with Barwick and Glasman \cite[\S13.5]{arXiv:1807.03281}.
We call it the \defn{condensed \category of coherent points} \smash{$ \Ptcohbf(\Xcal) $} of $ \Xcal $; it is given by sending an extremally disconnected profinite set $ K $ to the \category
\begin{equation*}
	\Ptcohbf(\Xcal)(K) \colonequals \Funupperstarcoh(\Xcal,\Sh(K))
\end{equation*}
of left exact left adjoints $ \fupperstar \colon \Xcal \to \Sh(K) $ that preserve coherent objects.
In particular, \smash{$ \Ptcohbf(\Xcal) $} is a full subcategory of $ \Ptbf(\Xcal) $.
Here are two examples to put these condensed \categories in context.

\begin{example}\label{ex:most_simple_examples_of_Ptbf}
	Let $ \Xcal $ is the \topos of sheaves on the Sierpiński space.
	Then for each extremally disconnected profinite set $ K $, the inclusion \smash{$ \Ptcohbf(\Xcal)(K) \inclusion \Ptbf(\Xcal)(K) $} is simply the inclusion
	\begin{equation*}
		\Clopen(K) \inclusion \Open(K)
	\end{equation*}
	of the poset of clopen subsets of $ K $ into the poset of all open subsets.

	More generally, if $ X $ is a spectral topological space and $ \Xcal = \Sh(X) $, then $ \Ptbf(\Xcal)(K) $ is the poset $ \Map(K,X) $ of continuous maps $ K \to X $ ordered by \textit{pointwise specialization:} $ f \leq g $ if and only if for each $ k \in K $, we have \smash{$ f(k) \in \overline{\{g(k)\}} $}.
	In this case, \smash{$ \Ptcohbf(\Xcal)(K) $} is the subposet
	\begin{equation*}
		\Map^{\qc}(K,X) \subset \Map(K,X)
	\end{equation*}
	of \textit{quasicompact} maps, i.e., maps such that the preimage of a quasicompact open is clopen.
\end{example}

\begin{example}\label{ex:model_theory_example_of_Ptbf}
	Let $ T $ be a coherent first-order theory, and let $ \Xcal $ be the classifying \topos of $ T $.
	Then for each extremally disconnected profinite set $ K $, the \category $ \Ptbf(\Xcal)(K) $ is the \category $ \Mod_{T}(\Sh(K;\Set)) $ of models of the theory $ T $ valued in the $ 1 $-category of sheaves of sets on $ K $.
	We write $ \Modbf_T \colonequals \Ptbf(\Xcal) $ and refer to $ \Modbf_T $ as the \textit{condensed category} of models of $ T $.
	The full subcategory \smash{$ \Ptcohbf(\Xcal)(K) $} has not traditionally been considered in model theory, but is a subcategory of models satisfying an additional finiteness property.
\end{example}

As we now explain, both algebraic geometry and model theory, these condensed \categories give rise to interesting invariants.
Write $ \Bup \colon \Catinfty \to \An $ for the left adjoint to the inclusion of anima into \categories; for \acategory $ \Ccal $, refer to $ \Bup \Ccal $ as the \textit{classifying anima} of $ \Ccal $.
Since $ \Bup $ preserves finite products, pointwise application of $ \Bup $ defined a left adjoint to the inclusion $ \Cond(\An) \inclusion \Cond(\Catinfty) $ of condensed anima into condensed \categories.
We also denote this left adjoint by $ \Bup $.

\medskip

\noindent \textbf{Algebraic geometry.} 
Let $ X $ be a qcqs scheme.
Building off our work with Barwick--Glasman \cite[\S13.5]{arXiv:1807.03281}, in recent work with Holzschuh--Lara--Mair--Martini--Wolf \cite{arXiv:2510.07443}, we studied the condensed classifying anima \smash{$ \Bup \Ptcohbf(X_{\et}) $} of the condensed \category of coherent points of the étale \topos of $ X $.
We call this condensed anima the \defn{condensed homotopy type} of $ X $, and denote it by $ \Picond(X) $.
One of our main results is that a mild completion of the fundamental group of $ \Picond(X) $ recovers the proétale fundamental group of Bhatt--Scholze \cite{MR3379634}.
See \cite[Theorem 1.9]{arXiv:2510.07443}.

\medskip

\noindent \textbf{Model theory.}
In forthcoming work with Damaj and Zhang \cite{Damaj-Haine-Zhang:Lascar_group}, we prove a similar result, but in the setting of model theory.
Given a complete first-order theory $ T $, in \cite{MR654786}, Lascar introduced a quasicompact topological group $ \GalL(T) $ now referred to as the \textit{Lascar group} of $ T $.
As the notation suggests, the Lascar group plays the role of the absolute Galois group of the theory $ T $.
Campion--Cousins--Ye \cite{MR4362919} proved that the discrete group underlying the Lascar group is the fundamental group of the classifying anima of the category $ \Mod_T(\Set) $ of models of $ T $ valued in the category of sets \cite[Theorem 3.5]{MR4362919}.
As explained in \Cref{ex:model_theory_example_of_Ptbf}, the condensed category of points of the classifying \topos of $ T $ is a natural enhancement of $ \Mod_T(\Set) $ to a condensed category.
In \cite{Damaj-Haine-Zhang:Lascar_group}, we refine Campion, Cousins, and Ye's result by showing that there is an isomorphism of condensed groups
\begin{equation*}
	\GalL(T) \equivalent \uppi_1(\Bup\Modbf_T)
\end{equation*}
between the Lascar group and the fundamental group of the condensed classifying anima of $ \Modbf_T $.

\medskip

The takeaway is that the proétale fundamental group and the Lascar group arise as special cases of \textit{similar} constructions.
But the first uses the smaller condensed \category \smash{$ \Ptcohbf(\Xcal) $}, and the second uses $ \Ptbf(\Xcal) $.
The main result of this note is that, in fact, these two invariants are special cases of the \textit{same} construction.%
\footnote{Up to the mild completion of $ \uppi_1(\Picond(X)) $ needed to recover the proétale fundamental group as defined by Bhatt--Scholze.}
Namely, there is a large class of coherent \topoi $ \Xcal $ for which the inclusion \smash{$ \Ptcohbf(\Xcal) \inclusion \Ptbf(\Xcal) $} induces an equivalence on condensed classifying anima.

The statement of our main result involves \textit{spectral \topoi} introduced in our work with Barwick and Glasman \cite{arXiv:1807.03281}. 
These are the bounded coherent \topoi $ \Xcal $ whose \category of points $ \Pt(\Xcal) $ satisfies the property that for every object $ \xupperstar \in \Pt(\Xcal) $, every endomorphism $ \xupperstar \to \xupperstar $ is an equivalence.
See \cref{subsec:spectral_topoi} for a quick review.
The most important example is that the étale \topos of a qcqs scheme is spectral.
Here is our main result:

\begin{theorem}[{(\Cref{thm:extra_adjoint_for_spectral_topoi})}]\label{intro_thm:extra_adjoint_for_spectral_topoi}
	Let $ \Xcal $ be a spectral \topos.
	Then for each extremally disconnected profinite set $ K $, the inclusion
	\begin{equation*}
		\Ptcohbf(\Xcal)(K) \inclusion \Ptbf(\Xcal)(K)
	\end{equation*}
	admits a left adjoint.
	As a consequence, the inclusion of condensed \categories $ \Ptcohbf(\Xcal) \inclusion \Ptbf(\Xcal) $ induces an equivalence on condensed classifying anima.
\end{theorem}

\begin{warning}
	If $ \Xcal $ is a coherent \topos which is not spectral, then even the inclusion 
	\begin{equation*}
		\Ptcohbf(\Xcal)(\pt) \inclusion \Ptbf(\Xcal)(\pt)
	\end{equation*}
	of coherent points into all points need not admit a left adjoint.
	For example, if $ \Xcal $ is the classifying \topos of the theory of abelian groups, this is the inclusion of the subcategory of finite abelian groups into all abelian groups.
\end{warning}

We conclude the introduction by unpacking the two simplest cases of \Cref{intro_thm:extra_adjoint_for_spectral_topoi}.

\begin{nul}
	As in \Cref{ex:most_simple_examples_of_Ptbf}, first consider the case when $ \Xcal $ is the \topos of sheaves on the Sierpiński space.
	Then the statement is that the inclusion $ \Clopen(K) \inclusion \Open(K) $ admits a left adjoint.
	This is clear: it is given by sending $ U \subset K $ to the closure \smash{$ \overline{U} $} (which is open since $ K $ is extremally disconnected).
\end{nul}

\begin{nul}
	More generally, if $ \Xcal $ is the \topos of sheaves on a spectral topological space $ X $, then \Cref{intro_thm:extra_adjoint_for_spectral_topoi} says that the inclusion
	\begin{equation*}
		\Map^{\qc}(K,X) \inclusion \Map(K,X)
	\end{equation*}
	of the poset of quasicompact maps into all maps admits a left adjoint.
	That is, any map $ f \colon K \to X $ has a best approximation by a quasicompact map that is a pointwise generization of $ f $.
	
	In this case, we can demonstrate our method of proof of \Cref{intro_thm:extra_adjoint_for_spectral_topoi}.
	By the Stone duality equivalence between distributive lattices and spectral spaces, $ \Map^{\qc}(K,X) $ is isomorphic to the poset
	\begin{equation*}
		\Map_{\Latt}(\Open^{\qc}(X),\Clopen(K))
	\end{equation*}
	of lattice homomorphisms from the poset of quasicompact opens in $ X $ to the poset of clopens in $ K $.
	By the equivalence between sober spaces and frames, $ \Map(K,X) $ is isomorphic to the poset 
	\begin{equation*}
		\Map_{\Frame}(\Open(X),\Open(K))
	\end{equation*}
	of frame maps $ \Open(X) \to \Open(K) $.
	Since quasicompact opens form a basis for the topology on $ X $ and are closed under finite intersections, restriction defines an isomorphism of posets
	\begin{equation*}
		\Map_{\Frame}(\Open(X),\Open(K)) \equivalence \Map_{\Latt}(\Open^{\qc}(X),\Open(K)) \period
	\end{equation*}
	Here, the target consists of lattice homomorphisms $ \Open^{\qc}(X) \to \Open(K) $.
	Under these identifications, the inclusion $ \Map^{\qc}(K,X) \inclusion \Map(K,X) $ is identified with the obvious inclusion
	\begin{equation}\label{eq:inclusion_of_lattice_maps}
		\Map_{\Latt}(\Open^{\qc}(X),\Clopen(K)) \inclusion \Map_{\Latt}(\Open^{\qc}(X),\Open(K)) \period
	\end{equation}
	To see that this admits a left adjoint, note that, again since $ K $ is extremally disconnected, the closure operator $ \Open(K) \to \Clopen(K) $ is left adjoint to the inclusion, and also preserves finite joins and the initial object.
	In an extremally disconnected space, for open subsets $ U $ and $ V $, we also have the identity $ \overline{U \intersect V} = \overline{U} \intersect \overline{V} $.
	Hence pointwise application of the closure operator $ \Open(K) \to \Clopen(K) $ defines a left adjoint to the inclusion \eqref{eq:inclusion_of_lattice_maps}.
\end{nul}


\subsection{Linear overview}

In \cref{sec:around_the_Los_ultraproduct_theorem}, we explain a sheaf-theoretic interpretation of the Łoś ultraproduct theorem.
This is implicit in Lurie's work on ultracategories \cite{Ultracategories}, and is what lets us analyze the condensed classifying anima of the condensed \category of points $ \Ptbf(\Xcal) $.
In \cref{sec:classifying_anima_of_condensed_categories_of_points}, we introduce the condensed \categories of points we consider in full detail, and prove \Cref{intro_thm:extra_adjoint_for_spectral_topoi}.
Since the results require a number of technical notions from the theory of \topoi, in \Cref{app:complements_on_topoi}, we've collected background on bounded and Postnikov complete \topoi, coherent \topoi and the classification of bounded coherent \topoi in terms of \pretopoi, and spectral \topoi.


\subsection{Notational conventions}

\begin{enumerate}
	\item We write $ \CATinfty $ for the \category of large \categories, and $ \AN \subset \CATinfty $ for the full subcategory spanned by the large anima (also referred to as spaces or \groupoids).
	We write $ \Catinfty \subset \CATinfty $ and $ \An \subset \AN $ for the full subcategories spanned by the small \categories and anima, respectively.

	\item We write $ \RTop $ for the \category of \topoi and geometric morphisms, i.e., right adjoints $ \flowerstar $ whose left adjoint $ \fupperstar $ is left exact.

	\item Given \categories $ \Xcal $ and $ \Ycal $ with limits and colimits, we write
	\begin{equation*}
		\Funlowerstar(\Ycal,\Xcal) \subset \Fun(\Ycal,\Xcal)
	\end{equation*}
	for the full subcategory spanned by the right adjoint functors whose left adjoint is left exact (when $ \Xcal $ and $ \Ycal $ are \topoi these are the geometric morphisms of \topoi).
	We write
	\begin{equation*}
		\Funupperstar(\Xcal,\Ycal) \subset \Fun(\Xcal,\Ycal)
	\end{equation*}
	for the full subcategory spanned by the left exact left adjoints (when $ \Xcal $ and $ \Ycal $ are \topoi these are the \defn{algebraic} morphisms of \topoi).

	\item Given \atopos $ \Xcal $ we write $ \Pt(\Xcal) \colonequals \Funupperstar(\Xcal,\An) $ for the \textit{\category of points} of $ \Xcal $.

	\item If $ \Xcal $ and $ \Ycal $ are \textit{coherent} \topoi (see \Cref{rec:coherence}), we write $ \Funupperstarcoh(\Xcal,\Ycal) \subset \Funupperstar(\Xcal,\Ycal)  $ for the full subcategory spanned by those left exact left adjoints that also preserve coherent objects. 

	\item Given \pretopoi $ \Ccal $ and $ \Dcal $ (see \Cref{def:pretopos}), we write $ \Funpre(\Ccal,\Dcal) \subset \Fun(\Ccal,\Dcal) $ for the full subcategory spanned by the morphisms of \pretopoi.
	That is, the functors that preserve finite limits, finite coproducts, and effective epimorphisms.
\end{enumerate}


\subsection{Acknowledgments}

We thank Clark Barwick for countless insightful conversations over the years around many of the topics appearing in this note.  
We thank Jacob Lurie for enlightening discussions on his work on ultracategories; in particular, for indicating that \Cref{cor:restriction_from_PtXbetaS_to_PtXS_admits_a_right_adjoint} should be true.


\section{Around the Łoś ultraproduct theorem}\label{sec:around_the_Los_ultraproduct_theorem}

In this section, we explain the following sheaf-theoretic formulation of the Łoś ultraproduct theorem: for any set $ S $ with Čech--Stone compactification $ \upbeta(S) $, the functor
\begin{equation*}
	\Sh(S) \inclusion \Sh(\upbeta(S))
\end{equation*}
given by pushforward along the inclusion $ S \inclusion \upbeta(S) $ preserves finite limits, finite coproducts, and effective epimorphisms.
See \Cref{cor:sheaf-theoretic_Los_ultraproduct_theorem}.
This formulation has the important consequence that for any coherent \topos $ \Xcal $, the restriction functor
\begin{equation*}
	\Funupperstar(\Xcal,\Sh(\upbeta(S))) \longrightarrow \Funupperstar(\Xcal,\Sh(S)) \equivalent \Pt(\Xcal)^S 
\end{equation*}
admits a fully faithful right adjoint (\Cref{cor:restriction_from_PtXbetaS_to_PtXS_admits_a_right_adjoint}).
This is what allows us to analyze the condensed classifying anima of the condensed \category $ \Ptbf(\Xcal) $ of points of $ \Xcal $ defined in the introduction

In \cref{subsec:ultraproducts_via_sheaf_theory}, we explain how to extract ultraproduct operations from the usual pushforward and pullback operations in sheaf theory.
This is well-known to experts, but we were unable to locate a reference.
In \cref{subsec:the_Los_ultraproduct_theorem_and_consequences}, we explain the sheaf-theoretic formulation of the Łoś ultraproduct theorem, its relation to the usual formulation, and some topos-theoretic consequences.


\subsection{Ultraproducts via sheaf theory}\label{subsec:ultraproducts_via_sheaf_theory}

We begin by recalling a bit of background on Čech--Stone compactifications.
We refer the unfamiliar reader to \cite[\S3.2]{Ultracategories} for details.

\begin{notation}
	Let $ S $ be a set.
	If we regard $ S $ as a topological space, it will always be with the discrete topology.
	We write $ j \colon S \inclusion \upbeta(S) $ for the Čech--Stone compactification of $ S $.
	We identify $ \upbeta(S) $ with the set of ultrafilters on $ S $.
	We typically denote an ultrafilter by $ \mu $.
	Under this identification, $ j $ sends an element $ s \in S $ to the principal ultrafilter generated by $ s $.
\end{notation}

\begin{recollection}\label{rec:clopens_of_betaS}
	Let $ S $ be a set.
	Then the natural map of posets
	\begin{equation*}
		\Clopen(\upbeta(S)) \to \Sub(S) \comma \quad U \mapsto U \intersect S
	\end{equation*}
	is an isomorphism.
	The inverse sends a subset $ S_0 \subset S $ to the closure $ \overline{S}_0 $ in $ \upbeta(S) $.
	Moreover, the closure \smash{$ \overline{S}_0 \subset \upbeta(S) $} is the image of the natural closed immersion $ \upbeta(S_0) \inclusion \upbeta(S) $. 
	The topology on $ \upbeta(S) $ is such that an ultrafilter $ \mu $ is in \smash{$ \overline{S}_0 $} if and only if the subset $ S_0 $ is contained in $ \mu $.
\end{recollection}

\begin{recollection}
	Let $ S $ be a set.
	There is a natural equivalence $ \Sh(S) \equivalent \An^S $.
	In one direction, this sends a sheaf $ X $ to the collection of stalks $ (X_s)_{s \in S} $.
	In the other direction, a collection of anima $ (X_s)_{s \in S} $ is sent to the sheaf given by the assignment
	\begin{equation*}
		S_0 \mapsto \prod_{s \in S_0} X_s \comma
	\end{equation*}
	with restriction maps the projections.
	(More formally, the right Kan extension of the functor $ S \to \An $ given by $ s \mapsto X_s $ along the inclusion $ S \inclusion \Sub(S)^{\op} $.)
	We often identify $ \Sh(S) $ with \smash{$ \An^S $} via this equivalence.
\end{recollection}

\begin{observation}[{(formula for $ \jlowerstar $ on clopens)}]\label{obs:formula_for_jlowerstar_on_clopens}
	Given a sheaf $ X \in \Sh(S) $, unpacking the formula for the pushforward of sheaves says that for each subset $ S_0 \subset S $, we have a natural identification
	\begin{equation*}
		\jlowerstar(X)(\overline{S}_0) \equivalent \prod_{s \in S_0} X_s \period
	\end{equation*}
\end{observation}

Now we fix our notation for ultraproducts. 

\begin{recollection}[(ultraproducts)]
	Let $ \Ecal $ be \acategory with filtered colimits and products, let $ S $ be a set, and let $ \mu \in \upbeta(S) $ be an ultrafilter.
	The \defn{ultraproduct functor} $ \int_{S} (-)\dmu \colon \Ecal^S \to \Ecal $ is the functor given by the assignment
	\begin{equation*}
		(X_s)_{s \in S} \mapsto \int_{S} X_s \dmu \colonequals \colim_{S_0 \in \mu^{\op}} \prod_{s \in S_0} X_s \period 
	\end{equation*}
	More formally, $ \int_{S} (-)\dmu $ is the composite
	\begin{equation*}
		\begin{tikzcd}[sep=4em]
			\Ecal^S \arrow[r] & \Fun(\Sub(S)^{\op},\Ecal) \arrow[r, "\mathrm{restrict}"] & \Fun(\mu^{\op},\Ecal) \arrow[r, "\colim"] & \Ecal \period
		\end{tikzcd}
	\end{equation*} 
	Here, the left-most functor is given by right Kan extension of the functor $ S \to \An $ given by $ s \mapsto X_s $ along the inclusion $ S \inclusion \Sub(S)^{\op} $.
\end{recollection}

The following is the sheaf-theoretic interpretation of ultraproducts.

\begin{proposition}\label{prop:stalks_of_jlowerstar}
	Let $ S $ be a set and let $ \mu \in \upbeta(S) $ be an ultrafilter on $ S $.
	Then the composite 
	\begin{equation*}
		\begin{tikzcd}
			\An^S \equivalent \Sh(S) \arrow[r, hooked, "\jlowerstar"] & \Sh(\upbeta(S)) \arrow[r, "\muupperstar"] & \An
		\end{tikzcd}
	\end{equation*}
	is equivalent to the ultraproduct functor $ \int_S (-)\dmu \colon \An^S \to \An $.
\end{proposition}

\noindent Said differently, the stalks of the sheaf $ \jlowerstar(X) $ are the ultraproducts $ \int_S X_s \dmu $.

\begin{proof}
	Let $ (X_s)_{s \in S} \in \An^S $ be an $ S $-indexed collection of anima corresponding to a sheaf $ X \in \Sh(S) $.
	Since the clopen subsets of $ \upbeta(S) $ form a basis for the topology on $ \upbeta(S) $, we see that we have natural equivalences
	\begin{align*}
		\muupperstar \jlowerstar(X) &= \colim_{U \ni \mu} \jlowerstar(X)(U) \\
		&\equivalent \colim_{\overline{S}_0 \ni \mu} \jlowerstar(X)(\overline{S}_0) && \textup{(\Cref{rec:clopens_of_betaS})} \\ 
		&\equivalent \colim_{\overline{S}_0 \ni \mu} \prod_{s \in S_0} X_s && \textup{(\Cref{obs:formula_for_jlowerstar_on_clopens})} \\
		&\equivalent \colim_{S_0 \in \mu^{\op}} \prod_{s \in S_0} X_s && \textup{(definition of the topology on $\upbeta(S)$)} \\
		&= \int_S X_s \dmu \period && \qedhere
	\end{align*}
\end{proof}


\subsection{The Łoś ultraproduct theorem \& consequences}\label{subsec:the_Los_ultraproduct_theorem_and_consequences}

In our mind, the following is the most fundamental formulation of the Łoś ultraproduct theorem:

\begin{proposition}[{(fundamental Łoś ultraproduct theorem, \SAG{Proposition}{E.3.3.8})}]\label{prop:fundamental_Los_ultraproduct_theorem}
	Let $ S $ be a set and let $ \mu \in \upbeta(S) $ be an ultrafilter on $ S $.
	Then the ultraproduct functor
	\begin{equation*}
		\int_S (-)\dmu \colon \An^S \to \An
	\end{equation*}
	is a morphism of \pretopoi, i.e., preserves finite limits, finite coproducts, and effective epimorphisms.
\end{proposition}

\noindent \Cref{prop:fundamental_Los_ultraproduct_theorem} has the following immediate consequence:

\begin{corollary}[(categorical logic formulation of the Łoś ultraproduct theorem)]\label{cor:categorical_logic_Los_ultraproduct_theorem}
	Let $ \Ccal $ be \apretopos (see \Cref{def:pretopos}).
	Then the full subcategory
	\begin{equation*}
		\Funpre(\Ccal,\An) \subset \Fun(\Ccal,\An)
	\end{equation*}
	spanned by the morphisms of \pretopoi is closed under the formation of ultraproducts.
\end{corollary}

\begin{remark}[(the usual formulation of the Łoś ultraproduct theorem)]\label{rem:the_usual_Los_ultraprpoduct_theorem}
	The usual Łoś ultraproduct theorem in logic can be deduced from \Cref{cor:categorical_logic_Los_ultraproduct_theorem}.
	Let us sketch how to do so.
	First, \Cref{cor:categorical_logic_Los_ultraproduct_theorem} implies the following $ 1 $-categorical statement: if $ \Ccal $ is a $ 1 $-pretopos, then the 
	full subcategory
	\begin{equation*}
		\Funpre(\Ccal,\Set) \subset \Fun(\Ccal,\Set)
	\end{equation*}
	of morphisms of $ 1 $-pretopoi (i.e., left exact functors that preserve effective epimorphisms and finite coproducts) is closed under ultraproducts.
	See \cite[Theorem 2.1.1]{Ultracategories} for this formulation.

	Let $ T $ be a coherent theory with classifying $ 1 $-topos $ \Set[T] $.
	Let $ \Syn_0(T) $ denote the full subcategory spanned by the coherent objects of $ \Set[T] $.
	Then $ \Syn_0(T) $ is a $ 1 $-pretopos, often referred to as the \textit{weak syntactic category} \cite[p. 2]{Lurie:Categorical_logic-2} or \textit{coherent syntactic category} \cites[Definition 1.4.1]{MR3752150} of $ T $.
	In this case, $ \Funpre(\Syn_0(T),\Set) $ is the $ 1 $-category $ \Mod_T $ of models of $ T $, and the categorical logic formulation of the Łoś ultraproduct theorem says that $ \Mod_T $ is closed under ultraproducts in the larger $ 1 $-category $ \Fun(\Syn_0(T),\Set) $.
	From this statement, it is relatively straightforward to prove the usual statement of the Łoś ultraproduct theorem; see \cites{MR900266}{nLab:Los_ultraproduct_theorem} for details.
\end{remark}

From the sheaf-theoretic description of ultraproducts (\Cref{prop:stalks_of_jlowerstar}), we can give an equivalent formulation of the fundamental Łoś ultraproduct theorem:

\begin{corollary}[(sheaf-theoretic formulation of the Łoś ultraproduct theorem)]\label{cor:sheaf-theoretic_Los_ultraproduct_theorem}
	For any set $ S $, the pushforward functor $ \jlowerstar \colon \Sh(S) \inclusion \Sh(\upbeta(S)) $ is a morphism of \pretopoi.
\end{corollary}

\begin{proof}
	Since the stalk functors $ (\muupperstar \colon \Sh(\upbeta(S)) \to \An)_{\mu \in \upbeta(S)} $ are jointly conservative and also morphisms of \pretopoi, $ \jlowerstar $ is a morphism of \pretopoi if and only if for each ultrafilter $ \mu \in \upbeta(S) $, the composite
	\begin{equation*}
		\muupperstar \jlowerstar \colon \An^S \equivalent \Sh(S) \to \An
	\end{equation*}
	is a morphism of \pretopoi.
	By \Cref{prop:stalks_of_jlowerstar}, the composite $ \muupperstar \jlowerstar $ is the ultraproduct functor $ \int_S (-)\dmu $.
	Thus the claim is equivalent to \Cref{prop:fundamental_Los_ultraproduct_theorem}.
\end{proof}

This sheaf-theoretic formulation of the Łoś ultraproduct theorem has pleasant features.
First, it does not explicitly refer to ultrafilters or ultraproducts (indeed, it does not reference a specific construction of the Čech--Stone compactification of $ S $).
Second, it has some interesting topos-theoretic consequences that aren't easy to deduce directly from \Cref{prop:fundamental_Los_ultraproduct_theorem} or \Cref{cor:categorical_logic_Los_ultraproduct_theorem}.
Let us now explain such a consequence.

\begin{corollary}\label{cor:extra_adjoint_for_models_of_a_pretopos}
	Let $ S $ be a set and let $ \Ccal $ be \apretopos.
	Then the adjunction
	\begin{equation*}
		\begin{tikzcd}[sep=4em]
			\Fun(\Ccal,\Sh(\upbeta(S))) \arrow[r, "\jupperstar \of -", shift left] & \Fun(\Ccal,\Sh(S)) \equivalent \Fun(\Ccal,\An)^S \arrow[l, "\jlowerstar \of -", shift left, hooked']
		\end{tikzcd}
	\end{equation*}
	restricts to an adjunction
	\begin{equation*}
		\begin{tikzcd}[sep=4em]
			\Funpre(\Ccal,\Sh(\upbeta(S))) \arrow[r, shift left] & \Funpre(\Ccal,\Sh(S)) \equivalent \Funpre(\Ccal,\An)^S \period \arrow[l, shift left, hooked']
		\end{tikzcd}
	\end{equation*}
\end{corollary}

\begin{proof}
	Immediate from \Cref{cor:sheaf-theoretic_Los_ultraproduct_theorem}.
\end{proof}

\begin{corollary}\label{cor:restriction_from_PtXbetaS_to_PtXS_admits_a_right_adjoint}
	Let $ \Xcal $ be a coherent \topos and let $ S $ be a set.
	Then the restriction functor
	\begin{equation*}
		\begin{tikzcd}[sep=3.5em]
			\Funupperstar(\Xcal,\Sh(\upbeta(S))) \arrow[r, "\jupperstar \of -"] & \Funupperstar(\Xcal,\Sh(S)) \equivalent \Pt(\Xcal)^S 
		\end{tikzcd}
	\end{equation*}
	admits a fully faithful right adjoint.
\end{corollary}

\begin{proof}
	Since $ \Sh(\upbeta(S)) $ and $ \Sh(S) $ are both bounded and Postnikov complete, by \Cref{prop:Funupperstar_to_a_bounded_Postnikov_complete_topos}, the above functor is identified with the functor 
	\begin{equation*}
		\begin{tikzcd}[sep=3.5em]
			\Funpre(\Xcal\cohtrun,\Sh(\upbeta(S))) \arrow[r, "\jupperstar \of -"]  & \Funpre(\Xcal\cohtrun,\Sh(S)) \period
		\end{tikzcd}
	\end{equation*}
	By \Cref{cor:extra_adjoint_for_models_of_a_pretopos}, this functor admits a fully faithful right adjoint given by post-composition with $ \jlowerstar \colon \Sh(S) \inclusion \Sh(\upbeta(S)) $.
\end{proof}

\begin{remark}
	Di Liberti \cite[\S1]{arXiv:2211.03104} and Di Liberti--Ye \cite[\S4.4]{arXiv:2504.16690} have studied $ 1 $-topoi satisfying the conclusion of \Cref{cor:restriction_from_PtXbetaS_to_PtXS_admits_a_right_adjoint}.
\end{remark}


\section{Classifying anima of condensed \texorpdfstring{$\infty$}{∞}-categories of points}\label{sec:classifying_anima_of_condensed_categories_of_points}

In this section, we prove the main result of this note (\Cref{thm:extra_adjoint_for_spectral_topoi}).
In order to do so, we begin in \cref{subsec:condensed_categories_of_points} by recalling the basics of the condensed \categories of points that we consider.
In \cref{subsec:classifying_anima_of_free_colimit_completions}, we record some facts about classifying anima that we need.
\Cref{subsec:the_case_of_spectral_topoi} proves our main result.


\subsection{Condensed \texorpdfstring{$\infty$}{∞}-categories of points}\label{subsec:condensed_categories_of_points}

We now define the two condensed \categories of points relevant to this note.
It is important that our condensed \categories are not just valued in large \categories, but in \textit{accessible} \categories.
So we begin by recalling some accessibility results in topos theory.

\begin{recollection}
	Let $ \Xcal $ and $ \Ycal $ be \topoi.
	Then 
	\begin{equation*}
		\Funupperstar(\Xcal,\Ycal) \subset \Fun(\Xcal,\Ycal)
	\end{equation*}
	is an accessible subcategory \HTT{Proposition}{6.3.1.13}.
	Moreover, since filtered colimits in \atopos commute with finite limits, $ \Funupperstar(\Xcal,\Ycal) $ is closed under filtered colimits in $ \Fun(\Xcal,\Ycal) $
	As a consequence, for all left exact left adjoints $ \fupperstar \colon \Wcal \to \Xcal $ and $ \gupperstar \colon \Ycal \to \Zcal $, the induced functors 
	\begin{equation*}
		\begin{tikzcd}[sep=3em]
			\Funupperstar(\Xcal,\Ycal) \arrow[r, "- \of \fupperstar"] & \Funupperstar(\Wcal,\Zcal)
		\end{tikzcd}
		\andeq
		\begin{tikzcd}[sep=3em]
			\Funupperstar(\Xcal,\Ycal) \arrow[r, "\gupperstar \of -"] & \Funupperstar(\Xcal,\Zcal)
		\end{tikzcd}
	\end{equation*}
	preserve filtered colimits (in particular, are accessible).
\end{recollection}

\begin{notation}\label{ntn:Acc}
	We write $ \Acc \subset \CATinfty $ for the non-full \category with objects accessible \categories and morphisms accessible functors.
	Let $ \Accfil \subset \Acc $ denote the non-full subcategory with objects accessible \categories that admit filtered colimits and morphisms functors that preserve filtered colimits.
\end{notation}

\begin{recollection}\label{rec:Acc_admits_limits}
	The \category $ \Acc $ admits limits and the forgetful functor $ \Acc \to \CATinfty $ preserves limits \kerodon{06LQ}.
	Similarly, the \category $ \Accfil $ admits limits and the forgetful functor $ \Accfil \to \CATinfty $ preserves limits.
\end{recollection}

\begin{notation}
	Write $ \Extr $ for the category of extremally disconnected profinite sets.
	Given \acategory $ \Ecal $ with finite products, the \category of \defn{condensed objects} of $ \Ecal $ is the full subcategory
	\begin{equation*}
		\Cond(\Ecal) \subset \Fun(\Extr^{\op},\Ecal)
	\end{equation*}
	spanned by the finite product-preserving presheaves.%
	\footnote{Of course, one needs to be careful about the usual subtleties with size issues. 
	We choose to deal with these by using universes.
	See \cite[Remark 2.36]{arXiv:2510.07443} for a detailed discussion about why how one chooses to deal with this does not affect the results in an essential way.}
\end{notation}

The first condensed \category of points we consider makes sense for \textit{any} \topos.
Recall that we write $ \RTop $ for the \category with objects \topoi and morphisms the geometric morphisms, i.e., right adjoints whose left adjoint is left exact.

\begin{definition}[{(condensed \category of points)}]\label{def:condensed_category_of_points}
	We write 
	\begin{align*}
		\Ptbf \colon \RTop &\longrightarrow \Cond(\Accfil) \\ 
	\shortintertext{for the functor}
		\Xcal &\longmapsto [K \mapsto \Funupperstar(\Xcal,\Sh(K))] \period
	\end{align*}
	We refer to $ \Ptbf(\Xcal) $ as the \defn{condensed \category of points} of $ \Xcal $.
\end{definition}

\begin{nul}
	Of course, the \category of global sections of $ \Ptbf(\Xcal) $ is simply the \category $ \Pt(\Xcal) $ of points of $ \Xcal $.
\end{nul}

\begin{remark}\label{obs:Ptbf_only_depends_on_bounded_reflection}
	Let $ \Xcal $ be a \topos.
	Then by \Cref{prop:Funupperstar_to_a_bounded_Postnikov_complete_topos}, the natural geometric morphism $ \Xcal \to \Xcal^{\b} $ to the bounded reflection (see \Cref{rec:bounded_topoi}) induces an equivalence
	\begin{equation*}
		\Ptbf(\Xcal) \equivalence \Ptbf(\Xcal^{\b}) \period
	\end{equation*}
\end{remark}

\begin{remark}[{(the \categorical enhancement of Lurie's work on ultracategories)}]
	In Lurie's work on ultracategories, he introduces the functor $ \Ptbf $, at least in the setting of $ 1 $-topoi.
	One of Lurie's main results is that when restricted to coherent $ 1 $-topoi and all geometric morphisms, this functor is fully faithful \cite[Remark 2.3.4, Theorem 4.3.3, \& Remark 4.3.4]{Ultracategories}.
	In light of \Cref{obs:Ptbf_only_depends_on_bounded_reflection}, the correct \categorical enhancement of this statement requires a bit more care.
	We expect that the functor
	\begin{equation*}
		\Ptbf \colon \left(\substack{\textup{bounded coherent \topoi} \\ \textup{and all geometric morphisms}} \right) \longrightarrow \Cond(\Accfil)
	\end{equation*}
	should be fully faithful (even as a functor of $ (\infty,2) $-categories).
	Moreover, this functor should factor through \acategorical enhancement of ultracategories and left ultrafunctors (formulated using \textit{ultracategory envelopes} as explained in \cite[\S8]{Ultracategories}).
	Restricting to $ 1 $-localic coherent \topoi would then recover Lurie's result.

	While many of Lurie's arguments work verbatim in the \categorical setting, there seem to be a few places where some nontrivial care is needed in order to generalize Lurie's proof.
\end{remark}

The second condensed \category of points we consider is the full subcategory of $ \Ptbf(\Xcal) $ spanned by the coherent geometric morphisms; it is only well-behaved for coherent \topoi.
We write \smash{$ \RTopcoh \subset \RTop $} for the subcategory with objects coherent \topoi and morphisms the coherent geometric morphisms.
In this case, the relevant accessibility result is that for a coherent \topos $ \Xcal $ and a profinite set $ K $, the \category
\begin{equation*}
	\Funupperstarcoh(\Xcal,\Sh(K))
\end{equation*}
of coherent algebraic morphisms $ \fupperstar \colon \Xcal \to \Sh(K) $ is small and idempotent complete (see \Cref{prop:Funupperstar_to_a_bounded_Postnikov_complete_topos}).

\begin{notation}
	Write $ \Catidem \subset \Catinfty $ for the full subcategory spanned by the idempotent complete \categories.
	Note that the small accessible \categories are exactly the small idempotent complete \categories \kerodon{06KS}, and every functor out of a small idempotent complete \category is accessible \kerodon{06KW}.
	So \smash{$ \Catidem $} is also a full subcategory of $ \Acc $.
\end{notation}

\begin{definition}[{(condensed \category of coherent points)}]\label{def:condensed_category_of_coherent_points}
	We write 
	\begin{align*}
		\Ptcohbf \colon \RTopcoh &\longrightarrow \Cond(\Catidem) \\ 
	\shortintertext{for the functor}
		\Xcal &\longmapsto [K \mapsto \Funupperstarcoh(\Xcal,\Sh(K))] \period
	\end{align*}
	We refer to \smash{$ \Ptcohbf(\Xcal) $} as the \defn{condensed \category of coherent points} of $ \Xcal $.
\end{definition}

\begin{remark}\label{obs:Ptcohbf_only_depends_on_bounded_reflection}
	Let $ \Xcal $ be a coherent \topos.
	Then by \Cref{prop:Funupperstar_to_a_bounded_Postnikov_complete_topos}, the natural geometric morphism $ \Xcal \to \Xcal^{\b} $ to the bounded reflection induces an equivalence
	\begin{equation*}
		\Ptcohbf(\Xcal) \equivalence \Ptcohbf(\Xcal^{\b}) \period
	\end{equation*}
\end{remark}

\begin{observation}
	Since the inclusions
	\begin{equation*}
		\Accfil \subset \Acc \andeq \Catidem \subset \Acc
	\end{equation*}
	both preserve limits, we can regard the functors $ \Ptbf $ and $ \Ptcohbf $ as valued in the \category $ \Cond(\Acc) $ of condensed accessible \categories.
	Hence for any bounded coherent \topos, by definition, there is a natural inclusion of condensed accessible \categories
	\begin{equation*}
		\Ptcohbf(\Xcal) \inclusion \Ptbf(\Xcal) \period
	\end{equation*}
\end{observation}

\begin{remark}[{(condensed \category of locally coherent points)}]
	There are also many \topoi of interest that are only locally coherent, but not coherent.
	For example, the étale \topos of a scheme that is not qcqs. 
	For this larger class of \topoi, it is better to consider the variant of \smash{$ \Ptcohbf $} that sends $ \Xcal $ to the condensed \category assigning an extremally disconnected profinite set $ K $ to the full subcategory of $ \Funupperstar(\Xcal,\Sh(K)) $ spanned by the \textit{locally coherent} algebraic morphisms.
\end{remark}


\subsection{Classifying anima of free colimit completions}\label{subsec:classifying_anima_of_free_colimit_completions}

We are interested in studying the condensed classifying anima of the condensed \categories $ \Ptbf(\Xcal) $ and \smash{$ \Ptcohbf(\Xcal) $}.
For various reasons, it is useful to know that these take values in small anima.
In order to explain this, we record some basic facts about classifying anima of free colimit completions.
In particular, that the classifying anima of an accessible \category is small (\Cref{cor:classifying_anima_of_free_cocompletions}).

\begin{recollection}
	We write $ \Bup \colon \CATinfty \to \AN $ for the left adjoint to the inclusion $ \AN \inclusion \CATinfty $.
	For \acategory $ \Ccal $, we call $ \Bup \Ccal $ the \defn{classifying anima} of $ \Ccal $.
	We say that a functor $ f \colon \Ccal \to \Dcal $ is a \defn{weak homotopy equivalence} if $ \Bup f $ is an equivalence.
	We say that \category $ \Ccal $ is \defn{weakly contractible} if $ \Bup\Ccal \equivalent \pt $.
	Finally, note that if $ \Ccal $ is a small \category, then $ \Bup \Ccal $ is also small.
\end{recollection}

Our first observation is that every anima admits limits and colimits indexed by weakly contractible \categories:

\begin{lemma}\label{lem:anima_admit_weakly_contractible_(co)limits}
	Let $ X $ be an anima and let $ \Ical $ be a weakly contractible \category.
	Then:
	\begin{enumerate}
		\item\label{lem:anima_admit_weakly_contractible_(co)limits.1} The anima $ X $ admits weakly contractible limits and colimits.

		\item\label{lem:anima_admit_weakly_contractible_(co)limits.2} Let $ f \colon \Ical \to \Ccal $ be a functor to any \category.
		If $ f $ admits a (co)limit, then every functor $ g \colon \Ccal \to X $ preserves this (co)limit.
	\end{enumerate}
\end{lemma}

\begin{proof}
	For \eqref{lem:anima_admit_weakly_contractible_(co)limits.1}, it suffices to show that the constant functor
	\begin{equation*}
		X \to \Fun(\Ical,X)
	\end{equation*}
	is an equivalence.
	Since $ \pt $ is an anima, the unique functor $ \Ical \to \pt $ factors as
	\begin{equation*}
		\begin{tikzcd}
			\Ical \arrow[r, "p"] & \Bup \Ical \arrow[r, "q"] & \pt \period
		\end{tikzcd}
	\end{equation*}
	Moreover, since $ \Ical $ is weakly contractible, $ q $ is an equivalence.
	Hence the constant functor factors as 
	\begin{equation*}
		\begin{tikzcd}
			X \arrow[r, "\qupperstar", "\sim"'{yshift=0.25ex}] & \Fun(\Bup \Ical,X) \arrow[r, "\pupperstar"] & \Fun(\Ical,X) \period
		\end{tikzcd}
	\end{equation*}
	Moreover, since $ X $ is an anima, $ \pupperstar $ is an equivalence.
	Hence the constant functor $ \pupperstar \qupperstar $ is an equivalence, as desired.

	For \eqref{lem:anima_admit_weakly_contractible_(co)limits.2}, note that since $ X \equivalent X^{\op} $, the claims for limits and colimits are dual.
	We prove the claim for colimits.  
	If $ f \colon \Ical \to \Ccal $ admits a colimit, and $ g \colon \Ccal \to X $ is any functor, we need to show that the natural map
	\begin{equation*}
		\textstyle \colim_{\Ical} gf \to g(\colim_{\Ical} f)
	\end{equation*} 
	is an equivalence in $ X $.
	But in $ X $, every map is an equivalence, so there is nothing to prove.
\end{proof}

\begin{notation}
	Let $ \Kcal $ be a collection of small \categories.
	Given a small \category $ \Ccal_0 $, let $ \PSh_{\Kcal}(\Ccal_0) $ be the free cocompletion of $ \Ccal_0 $ under colimits of diagrams indexed by \categories in $ \Kcal $.
	The existence of $ \PSh_{\Kcal}(\Ccal_0) $ follows from \HTT{Proposition}{5.3.6.2}.
	Explicitly, $ \PSh_{\Kcal}(\Ccal_0) $ can be constructed as the smallest full subcategory of presheaves of anima on $ \Ccal_0 $ containing the image of the Yoneda embedding and closed under colimits indexed by \categories in $ \Kcal $. 
\end{notation}

\begin{corollary}\label{cor:classifying_anima_of_free_cocompletions}
	Let $ \Ccal_0 $ be a small \category and let $ \Kcal $ be a collection of small weakly contractible \categories.
	Then the inclusion $ y \colon \Ccal_0 \inclusion \PSh_{\Kcal}(\Ccal_0) $ induces an equivalence
	\begin{equation*}
		\Bup \Ccal_0 \equivalence \Bup(\PSh_{\Kcal}(\Ccal_0)) \period
	\end{equation*}
	In particular, the anima $ \Bup(\PSh_{\Kcal}(\Ccal_0)) $ is small.
\end{corollary}

\begin{proof}
	First notice that the final claim follows from the fact that since $ \Ccal_0 $ is small, $ \Bup\Ccal_0 $ is also small.
	For the main claim, we prove that for every anima $ X $, the functor
	\begin{equation*}
		\begin{tikzcd}[row sep=2.5em, column sep=3em]
			\Fun(\Bup(\PSh_{\Kcal}(\Ccal_0)),X) \arrow[r, "-\of \Bup y"] & \Fun(\Bup\Ccal_0,X) 
		\end{tikzcd}
	\end{equation*}
	is an equivalence.
	Notice that by \Cref{lem:anima_admit_weakly_contractible_(co)limits}, $ X $ admits weakly contractible colimits and every functor $ \Dcal \to X $ preserves all weakly contractible colimits that $ \Dcal $ admits.
	Thus we see that the inclusion
	\begin{equation*}
		\Fun^{\Kcolim}(\PSh_{\Kcal}(\Ccal_0),X) \subset \Fun(\PSh_{\Kcal}(\Ccal_0),X)
	\end{equation*}
	of the full subcategory spanned by functors that preserve colimits indexed by \categories in $ \Kcal $ is an equality.
	Hence by the universal property of $ \PSh_{\Kcal}(\Ccal_0) $, restriction along the inclusion $ y \colon \Ccal_0 \inclusion \PSh_{\Kcal}(\Ccal_0) $ defines an equivalence
	\begin{equation*}
		\Fun(\PSh_{\Kcal}(\Ccal_0),X) \equivalence \Fun(\Ccal_0,X) \period
	\end{equation*}
	By adjunction we have a commutative square
	\begin{equation*}
		\begin{tikzcd}[row sep=2.5em, column sep=3em]
			\Fun(\PSh_{\Kcal}(\Ccal_0),X) \arrow[r, "-\of y", "\sim"'{yshift=0.25ex}] \arrow[d, "\wr"'{xshift=0.25ex}] & \Fun(\Ccal_0,X) \arrow[d, "\wr"{xshift=-0.25ex}] \\
			\Fun(\Bup(\PSh_{\Kcal}(\Ccal_0)),X) \arrow[r, "-\of \Bup y"'] & \Fun(\Bup\Ccal_0,X) \period
		\end{tikzcd}
	\end{equation*}
	Since all other functors are equivalences, the bottom horizontal functor is an equivalence. 
	Thus $ \Bup y \colon \Bup \Ccal_0 \to \Bup(\PSh_{\Kcal}(\Ccal_0)) $ is an equivalence, as desired.
\end{proof}

\begin{corollary}\label{cor:classifying_anima_of_an_accessible_category_is_small}
	If $ \Ccal $ is an accessible \category, then the classifying anima $ \Bup \Ccal $ is small.
\end{corollary}

\begin{proof}
	Since $ \Ccal $ is accessible, there exists a regular cardinal $ \kappa $ and small \category $ \Ccal_0 $ such that $ \Ccal $ is equivalent to the free cocompletion $ \Ind_\kappa(\Ccal_0) $ of $ \Ccal_0 $ under $ \kappa $-filtered colimits.
	Since $ \kappa $-filtered \categories are weakly contractible \cite[Tags \kerodonlink{02PJ} \& \kerodonlink{02QL}]{Kerodon}, the claim follows from \Cref{cor:classifying_anima_of_free_cocompletions}. 
\end{proof}

\begin{observation}[{(classifying anima of accessible \categories)}]
	Note that every small anima is idempotent complete, and by \Cref{lem:anima_admit_weakly_contractible_(co)limits}, every functor between small anima is accessible.
	Hence we have an inclusion $ \An \subset \Acc $.
	Since the classifying anima of an accessible \category is small (\Cref{cor:classifying_anima_of_an_accessible_category_is_small}), we deduce that the classifying anima functor $ \Bup \colon \CATinfty \to \AN $ restricts to a left adjoint 
	\begin{equation*}
		\Bup \colon \Acc \to \An
	\end{equation*}
	to the inclusion.
	Moreover, since $ \Acc \subset \CATinfty $ is closed under limits, $ \Bup \colon \Acc \to \An $ preserves finite products.
	Hence pointwise application of $ \Bup $ defines a left adjoint to the inclusion
	\begin{equation*}
		\Cond(\An) \inclusion \Cond(\Acc) \period
	\end{equation*}
	We also denote this left adjoint by $ \Bup \colon \Cond(\Acc) \to \Cond(\An) $.
	Given a condensed accessible \category $ \Ccal $, we refer to $ \Bup \Ccal $ as the \defn{condensed classifying anima} of $ \Ccal $.
\end{observation}


\subsection{The case of spectral \texorpdfstring{$\infty$}{∞}-topoi}\label{subsec:the_case_of_spectral_topoi}

We're now ready to prove the main result of this note.
The result applies to \textit{spectral \topoi} introduced in our work with Barwick and Glasman \cite{arXiv:1807.03281}.
The unfamiliar reader can consult \cref{subsec:spectral_topoi} for a quick review.

\begin{theorem}\label{thm:extra_adjoint_for_spectral_topoi}
	Let $ \Xcal $ be a spectral \topos.
	Then:
	\begin{enumerate}
		\item\label{thm:extra_adjoint_for_spectral_topoi.1} For each extremally disconnected profinite set $ K $, the inclusion \smash{$ \Ptcohbf(\Xcal)(K) \inclusion \Ptbf(\Xcal)(K) $} admits a left adjoint.

		\item\label{thm:extra_adjoint_for_spectral_topoi.2} The inclusion $ \Ptcohbf(\Xcal) \inclusion \Ptbf(\Xcal) $ induces an equivalence on condensed classifying anima.
	\end{enumerate}
\end{theorem}

\begin{proof}
	First note that since the classifying anima functor sends adjunctions to equivalences, \eqref{thm:extra_adjoint_for_spectral_topoi.2} is an immediate consequence of \eqref{thm:extra_adjoint_for_spectral_topoi.1}. 
	For \eqref{thm:extra_adjoint_for_spectral_topoi.1}, note that by Gleason's theorem \cites{MR0121775}[Chapter III, \S3.7]{MR861951} there exists a set $ S $ such that $ K $ is a retract of the Čech--Stone compactification $ \upbeta(S) $.
	Thus the inclusion \smash{$ \Ptcohbf(\Xcal)(K) \inclusion \Ptbf(\Xcal)(K) $} is a retract of the inclusion
	\begin{equation}\label{eq:inclusion_of_categories_of_points_for_betaS}
		\Ptcohbf(\Xcal)(\upbeta(S)) \inclusion \Ptbf(\Xcal)(\upbeta(S)) \period
	\end{equation}
	Since $ \Ptcohbf(\Xcal)(K) $ is idempotent complete, by \SAG{Lemma}{21.1.2.14} it suffices to show that the inclusion \eqref{eq:inclusion_of_categories_of_points_for_betaS} admits a left adjoint.
	That is, we're reduced to the case where $ K = \upbeta(S) $.

	In this case, we have a commutative triangle 
	\begin{equation*}
		\begin{tikzcd}[sep=4em]
			\Ptcohbf(\Xcal)(\upbeta(S)) = \Funupperstarcoh(\Xcal,\Sh(\upbeta(S))) \arrow[dr, "\sim"{sloped, yshift=-0.25em}] \arrow[d, hooked] & \\ 
			\Ptbf(\Xcal)(\upbeta(S)) = \Funupperstar(\Xcal,\Sh(\upbeta(S))) \arrow[r, "\jupperstar \of -"'] & \Funupperstar(\Xcal,\Sh(S)) \equivalent \Pt(\Xcal)^S \period
		\end{tikzcd}
	\end{equation*}
	Since $ \Xcal $ is spectral, the diagonal functor is an equivalence (see \Cref{lem:value_of_Ptcohbf_on_betaS}).
	Since $ \Xcal $ is coherent, by \Cref{cor:restriction_from_PtXbetaS_to_PtXS_admits_a_right_adjoint}, the bottom horizontal functor admits a fully faithful right adjoint $ \Pt(\Xcal)^S \inclusion \Ptbf(\Xcal)(\upbeta(S)) $.
	Thus, under the equivalence $ \Ptcohbf(\Xcal)(\upbeta(S)) \equivalence \Pt(\Xcal)^S $, the left-hand vertical inclusion is identified with the fully faithful right adjoint $ \Pt(\Xcal)^S \inclusion \Ptbf(\Xcal)(\upbeta(S)) $ to the horizontal functor.
	Hence the left-hand vertical functor also admits a left adjoint.
\end{proof}

\begin{warning}
	If $ K $ is a profinite set which is not extremally disconnected, then the inclusion
	\begin{equation*}
		\Funupperstarcoh(\Xcal,\Sh(K)) \inclusion \Funupperstar(\Xcal,\Sh(K))
	\end{equation*}
	generally does not admit a left adjoint.
	In particular, the inclusion $ \Ptcohbf(\Xcal) \inclusion \Ptbf(\Xcal) $ is not usually a right adjoint of condensed \categories in the sense of \cite[Definition 3.1.1 \& Proposition 3.2.9]{MR4752519}.
	To see this, let $ K =\NN^+ $ be the one-point compactification of the natural numbers and let $ \Xcal $ be the \topos of sheaves on the Sierpiński space.
	Then we're considering the inclusion
	\begin{equation}\label{eq:inclusion_of_clopens_of_NN+}
		\Clopen(\NN^+) \inclusion \Open(\NN^+)
	\end{equation}
	of clopen subsets of $ \NN^+ $ into all open subsets.
	
	To see that \eqref{eq:inclusion_of_clopens_of_NN+} does not admit a left adjoint, recall that a subset $ U \subset \NN^+ $ is clopen if and only if $ U $ is either finite and doesn't contain $ \infty $ or $ U $ is cofinite and does contain $ \infty $.
	Consider the subspace
	\begin{equation*}
		Y = 2\NN \union \{\infty\}
	\end{equation*}
	consisting of even natural numbers and $ \infty $; then $ Y $ is not clopen.
	If the inclusion \eqref{eq:inclusion_of_clopens_of_NN+} admitted a left adjoint $ L $, then there would be a smallest clopen subset $ L(Y) $ containing $ Y $.
	To see that no such clopen exists, note that for each odd natural number $ n $, the clopen $ U_n \colonequals \NN^+ \sminus \{n\} $ contains $ Y $.
	Hence
	\begin{equation*}
		L(Y) \subset \Intersection_{n \textup{ odd}} U_n = Y \period
	\end{equation*}
	Since $ Y \subset L(Y) $, this says that $ L(Y) = Y $; but this is impossible since $ Y $ is not clopen.
\end{warning}


\newpage

\appendix

\section{Complements on \texorpdfstring{$\infty$}{∞}-topoi}\label{app:complements_on_topoi}

In the main body of this note, we needed a few results from the theory of coherent \topoi that are not hard, but not easily citeable from the literature.
Since some parts of the theory get quite technical, for the convenience of the reader, we recall the necessities in this appendix.
In \cref{subsec:bounded_and_Postnikov_complete_topoi}, we recall background on bounded and Postnikov complete \topoi.
This is mostly not needed in this note, but we prove some technical lemmas that allow us to state a few results more cleanly in the main body.
In \cref{subsec:coherent_topoi_and_pretopoi}, we recall the basics of coherent \topoi and the classification of bounded coherent \topoi in terms of \pretopoi.
Finally, \Cref{subsec:spectral_topoi} recalls the theory of spectral \topoi introduced in our work with Barwick and Glasman \cite{arXiv:1807.03281}.


\subsection{Bounded and Postnikov complete \texorpdfstring{$\infty$}{∞}-topoi}\label{subsec:bounded_and_Postnikov_complete_topoi}

Now we recall two technical conditions on \topoi that make an appearance in our proofs: \textit{boundedness} and \textit{Postnikov completeness}.
They both guarantee that the \topos is controlled by its full subcategory of truncated objects, and are particular to the theory of \topoi (meaning there is no analogue in $ 1 $-topos theory).
The reader should refer to \cite[\SAGsec{A.7}]{SAG} for full details, or to \cite[Chapter 3]{arXiv:1807.03281} for a more detailed overview than the one provided here.

For boundedness, we first recall a bit about \textit{$ n $-localic \topoi}.
The idea is that $ n $-localic \topoi are the \topoi that are determined by their underlying $ n $-topoi of $ (n-1) $-truncated objects.

\begin{notation}
	Given \atopos $ \Xcal $ and integer $ n \geq -2 $, we write $ \Xcal_{\leq n} \subset \Xcal $ for the full subcategory spanned by the $ n $-truncated objects.
	The inclusion $ \Xcal_{\leq n} \inclusion \Xcal $ admits a left adjoint $ \trun_{\leq n} \colon \Xcal \to \Xcal_{\leq n} $ called \defn{$ n $-truncation}.
	We write $ \Xcal_{<\infty} \colonequals \Union_{n \geq -2} \Xcal_{\leq n} $ for the full subcategory spanned by the truncated objects.
\end{notation}

\begin{definition}[{\cite[\HTTsubsec{6.4.5}]{HTT}}]\label{def:n-localic_topoi}
	Let $ n \geq 0 $ be an integer.
	We say that \atopos $ \Xcal $ is \defn{$ n $-localic} if for every \topos $ \Ycal $, the natural functor
	\begin{equation*}
		\Funlowerstar(\Ycal,\Xcal) \to \Funlowerstar(\Ycal_{\leq n-1},\Xcal_{\leq n-1})
	\end{equation*}
	is an equivalence of \categories.
	The inclusion of the full subcategory of $ \RTop $ spanned by the $ n $-localic \topoi admits a left adjoint $ \Lup_n $.
	We call $ \Lup_n(\Xcal) $ the \defn{$ n $-localic reflection} of $ \Xcal $.
\end{definition}

\begin{nul}\label{nul:n-localic_in_terms_of_sites}
	The proof of \HTT{Proposition}{6.4.5.9} demonstrates that \atopos $ \Xcal $ is $n$-localic if and only if there exists a small $n$-site \textit{with all finite limits} $(\Ccal,\tau)$ and an equivalence $ \Xcal \equivalent \Sh_{\tau}(\Ccal)$.
\end{nul}

\begin{recollection}[{(boundedness)}]\label{rec:bounded_topoi}
	Let $ \Xcal $ be \atopos.
	The \defn{bounded reflection} of $ \Xcal $ is the cofiltered limit
	\begin{equation*}
		\Xcal^{\b} \colonequals \lim_{n \in \NN^{\op}} \Lup_n(\Xcal)
	\end{equation*}
	formed in $ \RTop $.
	We say that $ \Xcal $ is \defn{bounded} if the natural geometric morphism $ \blowerstar \colon \Xcal \to \Xcal^{\b} $ is an equivalence.
	The assignment $ \Xcal \mapsto \Xcal^{\b} $ is left adjoint to the inclusion of the full subcategory of $ \RTop $ spanned by the bounded \topoi \SAG{Proposition}{A.7.1.5}.
\end{recollection}

The second condition is defined in terms of the \categories $ \Xcal_{\leq n} $ of $ n $-truncated objects directly.
Note that these subcategories are not themselves \topoi.

\begin{recollection}[(Postnikov completeness)]\label{rec:Postnikov_completeness}
	Let $ \Xcal $ be \atopos.
	The \defn{Postnikov completion} of $ \Xcal $ is the limit 
	\begin{equation*}
		\Xcal^{\post} \colonequals  \lim \bigg(
		\begin{tikzcd}[sep=1.5em]
			\cdots \arrow[r] & \Xcal_{\leq n+1} \arrow[r, "\trun_{\leq n}"] & \Xcal_{\leq n} \arrow[r] & \cdots \arrow[r, "\trun_{\leq 0}"] & \Xcal_{\leq 0} 
		\end{tikzcd}\bigg)
		\index[notation]{Cpost@$C^{\post}$}
	\end{equation*}
	formed in \categories.
	We write $ \tupperstar \colon \Xcal \to \Xcal^{\post} $ for the natural comparison functor.
	We say that the \topos $ \Xcal $ is \defn{Postnikov complete} if $ \tupperstar \colon \Xcal \to \Xcal^{\post} $ is an equivalence.
	
	The Postnikov completion $ \Xcal^{\post} $ is also \atopos; moreover, $ \tupperstar \colon \Xcal \to \Xcal^{\post} $ is a left exact left adjoint \SAG{Theorem}{A.7.2.4}.
	The assignment $ \Xcal \mapsto \Xcal^{\post} $ is right adjoint to the inclusion of the full subcategory of $ \RTop $ spanned by the Postnikov complete \topoi 
	\SAG{Corollary}{A.7.2.6}.
	Moreover, the functors $ \Xcal \mapsto \Xcal^{\post} $ and $ \Xcal \mapsto \Xcal^{\b} $ define inverse equivalences between the full subcategories of $ \RTop $ spanned by the bounded and Postnikov complete \topoi \SAG{Corollary}{A.7.2.6}.
	
	In summary, if we write $ \RTop^{\b} $ and $ \RTop^{\post} $ for the full subcategories of $ \RTop $ spanned by the bounded and Postnikov complete coherent \topoi, respectively, we have the following diagram
	\begin{equation*}
		\begin{tikzcd}[column sep=4em, row sep=3em]
			\RTop^{\post} \arrow[dr, hooked, shift right] \arrow[rr, "\sim"'{yshift=0.25ex}, "{(-)^{\b}}", shift left] & & \RTop^{\b} \arrow[dl, hooked', shift left] \arrow[ll, shift left, "{(-)^{\post}}"{xshift=0.8ex}] \\
			& \RTop \arrow[ur, "{(-)^{\b}}"{xshift=1ex, yshift=-0.5ex}, shift left] \arrow[ul, "{(-)^{\post}}"'{xshift=-0.5ex, yshift=-0.5ex}, shift right] & \phantom{\RTop^{\b}} \period
		\end{tikzcd}
	\end{equation*}
	In particular, for any \topos $ \Xcal $, there are natural identifications
	\begin{equation*}
		\Xcal^{\b} = (\Xcal^{\post})^{\b} \andeq \Xcal^{\post} = (\Xcal^{\b})^{\post} \period
	\end{equation*}
\end{recollection}

\begin{observation}\label{obs:Xpost_Xbdd_and_X_have_the_same_truncated_objects}
	Let $ \Xcal $ be \atopos.
	Then the natural geometric morphisms
	\begin{equation*}
		\Xcal^{\post} \to \Xcal \to \Xcal^{\b}
	\end{equation*}
	restrict to equivalences 
	\begin{equation*}
		(\Xcal^{\post})_{<\infty} \equivalence \Xcal_{<\infty} \equivalence (\Xcal^{\b})_{<\infty}
	\end{equation*}
	on truncated objects.
	See \cite[\SAGthm{Lemma}{A.7.1.4} \& \SAGthm{Proposition}{A.7.3.7}]{SAG}.
\end{observation}

\begin{example}
	For any small \category $ \Ccal $, the \topos $ \PSh(\Ccal) $ of presheaves on $ \Ccal $ is Postnikov complete.
\end{example}

\begin{example}
	If $ X $ is a paracompact topological space of finite covering dimension or a spectral space of finite Krull dimension, then $ \Sh(X) $ is Postnikov complete \cites[\HTTthm{Corollary}{7.2.1.12}, \HTTthm{Theorem}{7.2.3.6} \& \HTTthm{Remark}{7.2.4.18}]{HTT}[Theorem 3.12]{MR4296353}.
	In particular, the \topos of sheaves on a profinite set is Postnikov complete.
\end{example}

\begin{example}[{\cite{MO:168526}}]
	Let $ X $ be a topological space.
	If $ X $ admits a CW structure, then the \topos $ \Sh(X) $ is Postnikov complete.
\end{example}


\subsection{Coherent \texorpdfstring{$\infty$}{∞}-topoi and \texorpdfstring{$\infty$}{∞}-pretopoi}\label{subsec:coherent_topoi_and_pretopoi}

We now recall the basics of coherent \topoi and the classification of bounded coherent \topoi in terms of \pretopoi.
The reader should refer to \cite[\SAGapp{A}]{SAG} for full details, or to \cite[Chapter 3]{arXiv:1807.03281} for a more detailed overview than the one provided here.

\begin{definition}[(coherence)]\label{rec:coherence}
	Let $ \Xcal $ be \atopos.
	We say that $\Xcal $ is \defn{$ 0 $-coherent} (or \defn{quasicompact}) if for every effective epimorphism $ e \colon \coprod_{i \in I} U_i \surjection 1_{\Xcal} $, there exists a finite subset $ I_0 \subset I $ such that the restriction $ \coprod_{i \in I_0} U_i \surjection 1_{\Xcal} $ of $ e $ is still an effective epimorphism.
	Let $ n \geq 0 $, and define $n$-coherence of \topoi and their objects recursively as follows.
	\begin{enumerate}
		\item An object $U\in \Xcal $ is \defn{$n$-coherent} if the \topos $\Xcal_{/U}$ is $n$-coherent.

		\item The \topos $\Xcal $ is \defn{locally $n$-coherent} if every object $U\in \Xcal $ admits a cover $\{\fromto{V_i}{U}\}_{i\in I}$ in which each $V_i$ is $n$-coherent.

		\item The \topos $\Xcal $ is \defn{$(n+1)$-coherent} if $ \Xcal $ is locally $n$-coherent, and the $n$-coherent objects of $\Xcal $ are closed under finite products.
	\end{enumerate}

	An \topos $\Xcal $ is \defn{coherent} if for each $ n \geq 0 $, the \topos $ \Xcal $ is $n$-coherent.
	An object $U$ of \atopos $\Xcal $ is \defn{coherent} if $\Xcal_{/U}$ is a coherent \topos.
	Finally, \atopos $\Xcal $ is \defn{locally coherent} if every object $U\in \Xcal $ admits a cover $\{\fromto{V_i}{U}\}_{i\in I}$ in which each $V_i$ is coherent.
\end{definition}

\begin{notation}
	Let $\Xcal $ be \atopos. 
	Write $ \Xcal^{\coh} \subset \Xcal $ for the full subcategory of $ \Xcal $ spanned by the coherent objects and $ \Xcal\cohtrun \subset \Xcal $ for the full subcategory of $ \Xcal $ spanned by the truncated coherent objects. 
\end{notation}

\begin{definition}[(coherent geometric morphism)]\label{def:coherent_geometric_morphism}
	A geometric morphism between coherent \topoi $ \flowerstar \colon \Ycal \to \Xcal $ is \defn{coherent} if, for each coherent object \smash{$ U \in \Xcal $}, the object $ \fupperstar(U) \in \Ycal $ is coherent.
	This is equivalent to the requirement that $ \fupperstar $ carries $ \Xcal\cohtrun $ to $ \Ycal\cohtrun $ \cite[Corollary 3.4.5]{arXiv:1807.03281}.
\end{definition}

\begin{recollection}[(coherence is detected on Postnikov completions and bounded reflections)]\label{rec:coherence_is_detected_on_Postnikov_completions_and_bounded_reflections}
	Given \atopos $ \Xcal $, the natural geometric morphisms $ \Xcal^{\post} \to \Xcal \to \Xcal^{\b} $ restrict to equivalences 
	\begin{equation*}
		(\Xcal^{\post})\cohtrun \equivalence \Xcal\cohtrun \equivalence (\Xcal^{\b})\cohtrun \period
	\end{equation*}
	Moreover, $ \Xcal $ is coherent if and only if $ \Xcal^{\post} $ is coherent if and only if $ \Xcal^{\b} $ is coherent.
	See \cite[Lemma 3.4.12]{arXiv:1807.03281}.
	In addition, if $ \Xcal $ and $ \Ycal $ are coherent, then a geometric morphism $ \flowerstar \colon \Ycal \to \Xcal $ is coherent if and only if $ \flowerstar^{\post} \colon \Ycal^{\post} \to \Xcal^{\post} $ is coherent if and only if $ \flowerstar^{\b} \colon \Ycal^{\b} \to \Xcal^{\b} $ is coherent \cite[3.4.13]{arXiv:1807.03281}
\end{recollection}

Some examples are in order.
The main one comes from sheaves on a finitary \site.

\begin{definition}\label{def:finitary_infty-site}
	\Asite $(\Ccal,\tau)$ is \defn{finitary} if $ \Ccal $ admits fiber products, and, for every object $ U \in \Ccal $ and every covering sieve $ S \subset \Ccal_{/U} $, there is a finite subset $\{U_i\}_{i\in I} \subset S $ that generates a covering sieve.
\end{definition}

\begin{proposition}[\SAG{Proposition}{A.3.1.3}]\label{prop:SAG.A.3.1.3}
	Let $(\Ccal,\tau)$ be a finitary \site.
	Then:
	\begin{enumerate}
		\item The \topos $\Sh_{\tau}(\Ccal)$ locally coherent.

		\item For every object $ U \in \Ccal $, the image of $ U $ under the sheafified Yoneda embedding is a coherent object of $\Sh_{\tau}(\Ccal)$.
	
		\item If, in addition, $ \Ccal $ admits a terminal object, then $ \Sh_{\tau}(\Ccal) $ is coherent.
	\end{enumerate}
\end{proposition}

\noindent Here are some more geometric examples that can be deduced from this.

\begin{example}
	If $ X $ is a sober topological space, then $ \Sh(X) $ is coherent if and only if $ X $ is spectral, i.e., additionally quasicompact, quasiseparated, and has a basis of quasicompact opens.
	In this case, $ \Sh(X)_{<\infty}^{\coh} $ is the full subcategory spanned by the constructible sheaves of anima on $ X $.
	If $ f \colon Y \to X $ is a map between spectral spaces, then the geometric morphism $ \flowerstar \colon \Sh(Y) \to \Sh(X) $ is coherent if and only if the map $ f $ is quasicompact.
\end{example}

\begin{example}
	A scheme $ X $ is quasicompact and quasiseparated if and only if its étale \topos $ X_{\et} $ is coherent.
	In this case, $ (X_{\et})_{<\infty}^{\coh} $ is the full subcategory spanned by the constructible étale sheaves of anima on $ X $.
	If $ f \colon Y \to X $ is any morphism between qcqs schemes, then the geometric morphism $ \flowerstar \colon Y_{\et} \to X_{\et} $ is coherent.
\end{example}

Let us now turn to explaining how \atopos that is both bounded and coherent is \emph{determined} by its truncated coherent objects.
The following is an axiomatization of the formal properties satisfied by the truncated coherent objects:

\begin{definition}[{(\pretopos)}]\label{def:pretopos}
	\Acategory $ \Ccal $ is an \defn{\pretopos} if:
	\begin{enumerate}
		\item The \category $ \Ccal $ admits finite limits.

		\item The \category $ \Ccal $ admits finite coproducts, which are universal and disjoint.

		\item Groupoid objects in $ \Ccal $ are effective, and their geometric realizations are universal.
	\end{enumerate}
	If $ \Ccal $ and $ \Dcal $ are \pretopoi, then a functor $ \fupperstar \colon \Ccal \to \Dcal $ is a \emph{morphism of \pretopoi} if $ \fupperstar $ preserves finite limits, finite coproducts, and effective epimorphisms.
\end{definition}

\begin{notation}
	We write $ \Pretop \subset \CATinfty $ for the subcategory consisting of \pretopoi and morphisms of \pretopoi.
	Given \pretopoi $ \Ccal $ and $ \Dcal $, we write
	\begin{equation*}
		\Funpre(\Ccal,\Dcal) \subset \Fun(\Ccal,\Dcal)
	\end{equation*}
	for the full subcategory spanned by the morphisms of \pretopoi.
\end{notation}

\begin{example}[{\SAG{Corollary}{A.6.1.7}}]
	Every \topos is \apretopos.
	If $\Xcal $ is a coherent \topos, then the full subcategory $ \Xcal^{\coh} \subset \Xcal $ spanned by the coherent objects is \apretopos.
\end{example}

\begin{definition}\label{def:boundedpretopos}
	\Apretopos $ \Ccal $ is \defn{bounded} if $ \Ccal $ is small and every object of $ \Ccal $ is truncated.
	We write
	\begin{equation*}
		\Pretopb \subset \Pretop
	\end{equation*}
	for the full subcategory spanned by the bounded \pretopoi.
\end{definition}

\begin{example}
	If $ \Xcal $ is a coherent \topos, then the full subcategory $ \Xcal\cohtrun $ is a bounded \pretopos.
\end{example}

In order to state the key classification theorem for bounded coherent \topoi, we need to fix some notation.
The first is for a natural Grothendieck topology on any \pretopos.

\begin{notation}[(effective epimorphism topology)]\label{ntn:effective_epimorphism_topology}
	Let $ \Ccal $ be \apretopos.
	We write $ \eff $ for the topology on $ \Ccal $ where a sieve $ S $ on $ X \in \Ccal $ is covering if and only if there exist finitely many object $ U_1,\ldots,U_n \in S $ such that the induced map $ U_1 \coproduct \cdots \coproduct U_n \to X $ is an effective epimorphism.
	This is a finitary topology that we refer to as the \defn{effective epimorphism topology} \cite[\SAGsubsec{A.6.2}]{SAG}.
	Importantly, the effective epimorphism topology on \apretopos is a subcanonical topology \SAG{Corollary}{A.6.2.6}.
\end{notation}

\begin{notation}\label{ntn:RTopcoh}
	We write $ \RTopcoh \subset \RTop $ for the subcategory whose objects are coherent \topoi and whose morphisms are coherent geometric morphisms.
	Given coherent \topoi $ \Xcal $ and $ \Ycal $, we write
	\begin{equation*}
		\Funupperstarcoh(\Xcal,\Ycal) \subset \Funupperstar(\Xcal,\Ycal)
	\end{equation*}
	for the full subcategory spanned by those algebraic morphisms $ \fupperstar \colon \Xcal \to \Ycal $ that preserve coherent objects, i.e., the \textit{coherent} algebraic morphisms.
	We write $ \RTopbc \subset \RTopcoh $ for the full subcategory spanned by those coherent \topoi that are also bounded.
\end{notation}

\begin{theorem}[{(classification of bounded coherent \topoi, \SAG{Theorem}{A.7.5.3})}]\label{thm:classification_of_bounded_coherent_topoi}
	The constructions
	\begin{equation*}
		\goesto{\Xcal}{\Xcal\cohtrun} \andeq \goesto{\Ccal}{\Sheff(\Ccal)}
	\end{equation*}
	are mutually inverse equivalences of \categories
	\begin{equation*}
		\RTopbc \simeq \Pretop^{\b,\op} \period
	\end{equation*}
	Moreover, for bounded coherent \topoi $ \Xcal $ and $ \Ycal $, restriction along the inclusion $ \Xcal\cohtrun \inclusion \Xcal $ defines an equivalence if \categories
	\begin{equation*}
		\Funupperstarcoh(\Xcal,\Ycal) \equivalence \Funpre(\Xcal\cohtrun,\Ycal\cohtrun) \period
	\end{equation*}
\end{theorem}

We conclude this subsection by collecting a few technical results that we could not find references for.
They're concerned with describing arbitrary algebraic morphisms from a coherent \topos in terms of truncated coherent objects, and the idempotent completeness of \categories of coherent algebraic morphisms.
First we record what happens in the bounded case:

\begin{proposition}\label{prop:algebraic_morphisms_from_a_bounded_coherent_topos_to_any_topos}
	Let $ \Xcal $ and $ \Ycal $ be \topoi.
	If $ \Xcal $ is bounded coherent, then restriction along the inclusion $ \Xcal\cohtrun \inclusion \Xcal $ defines an equivalence of \categories
	\begin{equation*}
		\Funupperstar(\Xcal,\Ycal) \equivalence \Funpre(\Xcal\cohtrun,\Ycal) \period
	\end{equation*}
\end{proposition}

\begin{proof}
	Since $ \Xcal $ is bounded coherent, by \Cref{thm:classification_of_bounded_coherent_topoi} we have $ \Xcal \equivalent \Sheff(\Xcal\cohtrun) $.
	Hence the claim follows from \SAG{Proposition}{A.6.4.4}. 
\end{proof}

\begin{lemma}\label{lem:coherent_objects_of_a_bounded_coherent_topos_are_closed_under_retracts}
	Let $ \Xcal $ be a bounded coherent \topos.
	Then the full subcategory $ \Xcal\cohtrun \subset \Xcal $ is closed under retracts.
	In particular, $ \Xcal\cohtrun $ is a small idempotent complete \category.
\end{lemma}

\begin{proof}
	See the proof of \SAG{Corollary}{A.7.5.4}.
\end{proof}

\begin{lemma}\label{lem:coherent_geometric_morphisms_are_closed_under_retracts}
	Let $ \Xcal $ and $ \Ycal $ be bounded coherent \topoi.
	Then the full subcategory
	\begin{equation*}
		\Funupperstarcoh(\Xcal,\Ycal) \subset \Funupperstar(\Xcal,\Ycal)
	\end{equation*}
	is closed under retracts.
	Hence, $ \Funupperstarcoh(\Xcal,\Ycal) $ is a small idempotent complete \category.
\end{lemma}

\begin{proof}
	By the classification of bounded coherent \topoi, 
	\begin{equation*}
		\Funupperstarcoh(\Xcal,\Ycal) \equivalent \Funpre(\Xcal\cohtrun,\Ycal\cohtrun) \period
	\end{equation*}
	Since both $ \Xcal\cohtrun $ and $ \Ycal\cohtrun $ are small \categories, $ \Funupperstarcoh(\Xcal,\Ycal) $ is also a small \category.

	For the statement about closure under retracts (which immediately implies idempotent completeness), let $ \fupperstar \colon \Xcal \to \Ycal $ be an algebraic morphism that is a retract in $ \Funupperstar(\Xcal,\Ycal) $ of an algebraic morphism $ \gupperstar \colon \Xcal \to \Ycal $ that preserves truncated coherent objects.
	Then for each truncated coherent object $ X \in \Xcal $, the object $ \fupperstar(X) $ is a retract of the truncated coherent object $ \gupperstar(X) $.
	\Cref{lem:coherent_objects_of_a_bounded_coherent_topos_are_closed_under_retracts} then implies that $ \fupperstar(X) $ is truncated coherent; that is $ \fupperstar $ preserves truncated coherent objects, as desired.
\end{proof}

For the condensed \categories of points considered in this note, we're interested in (coherent) geometric morphisms from an arbitrary coherent \topos to the \topos of sheaves on a profinite set. 
The latter is bounded and Postnikov complete.
In this situation, variants of \Cref{prop:algebraic_morphisms_from_a_bounded_coherent_topos_to_any_topos,lem:coherent_geometric_morphisms_are_closed_under_retracts} hold without the boundedness assumption on $ \Xcal $.

\begin{proposition}\label{prop:Funupperstar_to_a_bounded_Postnikov_complete_topos}
	Let $ \Xcal $ and $ \Ycal $ be \topoi.
	If $ \Ycal $ is bounded and Postnikov complete, then: 
	\begin{enumerate}
		\item\label{prop:Funupperstar_to_a_bounded_Postnikov_complete_topos.1} The natural geometric morphism $ \Xcal \to \Xcal^{\b} $ induces an equivalence
		\begin{equation*}
			\Funupperstar(\Xcal,\Ycal) \equivalence \Funupperstar(\Xcal^{\b},\Ycal) \period
		\end{equation*}

		\item\label{prop:Funupperstar_to_a_bounded_Postnikov_complete_topos.2} If $ \Xcal $ is coherent, then restriction along the inclusion $ \Xcal\cohtrun \inclusion \Xcal $ defines an equivalence of \categories
		\begin{equation*}
			\Funupperstar(\Xcal,\Ycal) \equivalence \Funpre(\Xcal\cohtrun,\Ycal) \period
		\end{equation*}

		\item\label{prop:Funupperstar_to_a_bounded_Postnikov_complete_topos.3} If $ \Xcal $ and $ \Ycal $ are coherent, then the natural geometric morphism $ \Xcal \to \Xcal^{\b} $ induces an equivalence
		\begin{equation*}
			\Funupperstarcoh(\Xcal,\Ycal) \equivalence \Funupperstarcoh(\Xcal^{\b},\Ycal) \period
		\end{equation*}
		Moreover, $ \Funupperstarcoh(\Xcal,\Ycal) $ is a small idempotent complete \category.
	\end{enumerate}
\end{proposition}

\begin{proof}
	For \eqref{prop:Funupperstar_to_a_bounded_Postnikov_complete_topos.1}, notice that since $ \Ycal $ is both bounded and Postnikov complete by \Cref{rec:Postnikov_completeness} we have natural equivalences
	\begin{align*}
		\Funupperstar(\Xcal,\Ycal) &\equivalent \Funupperstar(\Xcal^{\post},\Ycal) \\
		&\equivalent \Funupperstar((\Xcal^{\post})^{\b},\Ycal) \\
		&\equivalent \Funupperstar(\Xcal^{\b},\Ycal) \period
	\end{align*}

	For \eqref{prop:Funupperstar_to_a_bounded_Postnikov_complete_topos.2}, note that since the natural geometric morphism $ \Xcal \to \Xcal^{\b} $ restricts to an equivalence on truncated coherent objects, the claim follows from \eqref{prop:Funupperstar_to_a_bounded_Postnikov_complete_topos.1} and \Cref{prop:algebraic_morphisms_from_a_bounded_coherent_topos_to_any_topos}.
	Similarly, the equivalence in \eqref{prop:Funupperstar_to_a_bounded_Postnikov_complete_topos.3} follows from the fact that the natural geometric morphism $ \Xcal \to \Xcal^{\b} $ restricts to an equivalence on truncated coherent objects, item \eqref{prop:Funupperstar_to_a_bounded_Postnikov_complete_topos.2}, and \Cref{thm:classification_of_bounded_coherent_topoi}.
	The statement that $ \Funupperstarcoh(\Xcal,\Ycal) $ is small and idempotent complete then follows from \Cref{lem:coherent_geometric_morphisms_are_closed_under_retracts}.
\end{proof}


\subsection{Spectral \texorpdfstring{$\infty$}{∞}-topoi}\label{subsec:spectral_topoi}

We now briefly review the theory of \textit{spectral \topoi} introduced in our work with Barwick and Glasman \cite{arXiv:1807.03281}.
In addition to bounded coherence, spectrality asks for an additional condition on the \category of points of \atopos.

\begin{recollection}\label{rec:layered_categories}
	\Acategory $ \Ccal $ is \defn{layered} if for each object $ x \in \Ccal $, every endomorphism $ x \to x $ is an equivalence.
	Equivalently, $ \Ccal $ is layered if and only if there exists a poset $ P $ and a conservative functor $ \Ccal \to P $.
	We write $ \Lay \subset \Catinfty $ for the full subcategory spanned by the layered \categories.
\end{recollection}

\begin{recollection}[{(spectral \topoi)}]\label{rec:spectral_topoi}
	\Atopos $ \Xcal $ is \defn{spectral} if $ \Xcal $ is bounded coherent and the \category of points $ \Pt(\Xcal) $ is layered.
	If $ \Xcal $ is a spectral \topos, then every point of $ \Xcal $ is coherent.
	We write \smash{$ \RTopspec \subset \RTopcoh $} for the full subcategory spanned by the spectral \topoi.
\end{recollection}

\noindent Here's the most important example:

\begin{example}
	If $ X $ is a qcqs scheme, then the étale \topos of $ X $ is spectral.
\end{example}

One of the main results is that spectral \topoi admit an even more simple classification than bounded coherent \topoi, in terms of \textit{profinite} layered \categories.
Here's the correct notion of finiteness:

\begin{recollection}
	An anima $ X $ is \defn{\pifinite} if $ \uppi_0(X) $ is finite, $ X $ is truncated, and for each integer $ n \geq 1 $ and point $ x \in X $, the group $ \uppi_n(X,x) $ is finite.
	\Acategory $ \Ccal $ is \defn{\pifinite} if $ \Ccal $ has finitely many objects up to equivalence and all mapping anima are \pifinite.
	We write $ \Catfin \subset \Catinfty $ for the full subcategory spanned by the \pifinite \categories, and $ \Layfin \subset \Catfin $ for the full subcategory spanned by the \pifinite layered \categories. 
\end{recollection}

\begin{example}
	If $ \Ccal $ is a \pifinite layered \category, then there is a natural equivalence
	\begin{equation*}
		\Ccal \equivalence \Pt(\Fun(\Ccal,\An)) \period
	\end{equation*}
	Moreover, the \topos $ \Fun(\Ccal,\An) $ is spectral.
\end{example}

\begin{recollection}[{(\categorical Hochster duality)}]
	The assignment $ \Ccal \mapsto \Fun(\Ccal,\An) $ with functoriality given by right Kan extension defines a left exact functor
	\begin{equation*}
		\Layfin \to \RTopcoh \period
	\end{equation*}
	The \category $ \RTopcoh $ admits limits, so by the universal property of \proobjects, this functor extends to a limit-preserving functor
	\begin{equation}\label{eq:embedding_of_profinite_layered_categories_into_coherent_topoi}
		\Pro(\Layfin) \to \RTopcoh \period
	\end{equation}
	One of the main results of our work with Barwick and Glasman is that this functor is fully faithful with image $ \RTopspec $ \cite[Theorem 9.3.1]{arXiv:1807.03281}.
	We refer to this result as \textit{\categorical Hochster duality}; it generalizes Hochster's equivalence between the category of \proobjects in the category of finite posets and the category of spectral spaces and quasicompact maps \cites{Hochster:primeideal}{Hochster:thesis}.
	
	Even better, the fully faithful functor \eqref{eq:embedding_of_profinite_layered_categories_into_coherent_topoi} is a right adjoint.
	We denote the left adjoint by
	\begin{equation*}
		\widehat{\Pi}_{(\infty,1)} \colon \RTopcoh \to \Pro(\Layfin) \comma
	\end{equation*}
	and for a coherent \topos $ \Xcal $, refer to $ \widehat{\Pi}_{(\infty,1)}(\Xcal) $ as the \defn{profinite stratified shape} of $ \Xcal $.
	The composite
	\begin{equation*}
		\begin{tikzcd}[sep=3.5em]
			\RTopspec \arrow[r, "\widehat{\Pi}_{(\infty,1)}", "\sim"'] & \Pro(\Layfin) \arrow[r, "\lim"] & \Catinfty
		\end{tikzcd}
	\end{equation*}
	is equivalent to the functor sending a spectral \topos $ \Xcal $ to its \category $ \Pt(\Xcal) $ of points.
\end{recollection}

\begin{recollection}[{(relationship to condensed mathematics)}]\label{rec:relationship_to_condensed_math}
	The constant sheaf functor
	\begin{equation*}
		\Catfin \inclusion \Cond(\Catinfty)
	\end{equation*}
	extends along cofiltered limits to a fully faithful functor
	\begin{equation*}
		\Pro(\Catfin) \inclusion \Cond(\Catinfty) \period
	\end{equation*}
	By \categorical Hochster duality, the composite
	\begin{equation*}
		\begin{tikzcd}[sep=3.5em]
			\RTopspec \arrow[r, "\widehat{\Pi}_{(\infty,1)}", "\sim"'] & \Pro(\Layfin) \arrow[r, hooked] & \Pro(\Catfin) \arrow[r, hooked] & \Cond(\Catinfty)
		\end{tikzcd}
	\end{equation*}
	is given by the assignment
	\begin{equation*}
		\Xcal \mapsto [K \mapsto \Funupperstarcoh(\Xcal,\Sh(K))] \period
	\end{equation*}
	That is, the composite is the functor $ \Ptcohbf $ of \Cref{def:condensed_category_of_coherent_points}. 
	As a result, 
	\begin{equation*}
		\Ptcohbf \colon \RTopspec \to \Cond(\Catinfty)
	\end{equation*}
	is fully faithful.
\end{recollection}

The key property of spectral \topoi that we need to prove our main result (\Cref{thm:extra_adjoint_for_spectral_topoi}) is that the value of \smash{$ \Ptcohbf(\Xcal) $} on a Čech--Stone compactification is particularly simple:

\begin{lemma}\label{lem:value_of_Ptcohbf_on_betaS}
	Let $ \Xcal $ be a spectral \topos and let $ S $ be a set.
	Then the restriction functor
	\begin{equation*}
		\begin{tikzcd}[sep=3.5em]
			\Funupperstarcoh(\Xcal,\Sh(\upbeta(S))) \arrow[r, "\jupperstar \of -"] & \Funupperstar(\Xcal,\Sh(S)) \equivalent \Pt(\Xcal)^S 
		\end{tikzcd}
	\end{equation*}
	is an equivalence of \categories.
\end{lemma}

\begin{proof}
	There are two ways to see this; both use that every point of a spectral \topos is coherent (see \Cref{rec:spectral_topoi}).
	The first is to appeal to the fact that for every condensed \category $ \Ccal $ in the image of the fully faithful embedding $ \Pro(\Catfin) \inclusion \Cond(\Catinfty) $, the restriction map 
	\begin{equation*}
		\Ccal(\upbeta(S)) \longrightarrow \Ccal(S) = \prod_{s \in S} \Ccal(\{s\})
	\end{equation*}
	is an equivalence of \categories \cite[Proposition 2.22]{arXiv:2510.07443}.
	The claim then follows from \Cref{rec:relationship_to_condensed_math}.

	Here's a second, more direct, proof.
	The limit functor $ \lim \colon \Pro(\Catfin) \to \Catinfty $ admits a left adjoint $ \Ccal \mapsto \Ccal_{\uppi}^{\wedge} $, that we refer to as \textit{profinite completion}.
	For any set $ S $, the profinite completion $ S_{\uppi}^{\wedge} $ is simply the profinite set $ \upbeta(S) $ \SAG{Remark}{E.5.2.6}.
	So the claim follows by combining \categorical Hochster duality with the universal property of profinite completion.
\end{proof}


\DeclareFieldFormat{labelnumberwidth}{#1}
\printbibliography[keyword=alph, heading=references]
\DeclareFieldFormat{labelnumberwidth}{{#1\adddot\midsentence}}
\printbibliography[heading=none, notkeyword=alph]

\end{document}